\documentclass[twoside,11pt]{article}
\usepackage{latexsym}
\usepackage[english]{babel}
\usepackage{amssymb,amsthm,epsf}

\newcommand{\oR}{{\mathbb R}}
\newcommand{\oN}{{\mathbb N}}
\newcommand{\oZ}{{\mathbb Z}}

\newcommand{\EE}{{\mathbb E}}
\newcommand{\PP}{{\mathbb P}}

\newtheorem{defn}{Definition}
\newtheorem{res}{Lemma}
\newtheorem{thm}{Theorem}

\newtheorem{ex}{Example}
\renewenvironment{proof}{\noindent{\bf Proof:} }{\hfill $\square$ \\}

\begin{document}
\begin{center}
{\Large\sc A Spectral Mean for Point Sampled Closed Curves}\\[.5in] 

\noindent
{\large M.N.M. van Lieshout}\\[.2in]
\noindent
{\em CWI \\
Science Park 123, 1098 XG Amsterdam, The Netherlands} 
\end{center}

\begin{verse}
{\footnotesize
\noindent
{\bf Abstract}\\
\noindent
We propose a spectral mean for closed curves described by sample
points on its boundary subject to mis-alignment and noise. 
First, we ignore mis-alignment and derive maximum likelihood 
estimators of the model and noise parameters in the Fourier domain.
We estimate the unknown curve by back-transformation and derive
the distribution of the integrated squared error. Then, we model
mis-alignment by means of a shifted parametric diffeomorphism and
minimise a suitable objective function simultaneously over the unknown 
curve and the mis-alignment parameters. Finally, the method is
illustrated on simulated data as well as on photographs of Lake Tana 
taken by astronauts during a Shuttle mission.
\\[0.2in]

\noindent
{\em Keywords \& Phrases:}
alignment, cyclic Gaussian process, diffeomorphism, flow,
integrated squared error, Jordan curve, spectral analysis.

}

\begin{center}
{\em In memory of J. Harrison.}
\end{center}
\end{verse}

\section{Introduction}

Many geographical or biological objects are observed in image form.
The boundaries of such objects are seldom crisp due to measurement 
error and discretisation, or because the boundaries themselves are 
intrinsically indeterminite \cite{BurrFran96}. Moreover, the objects
are not static so that if multiple images are taken, the object may 
have been deformed. This can be due, for example, to patient movements
in medical imagery of organs, or to external influences such as flooding
in remotely sensed images of rivers or lakes.

One attempt to model natural objects under uncertainty is fuzzy set 
theory (see e.g.\ \cite{Zimm01})). However, the underlying axioms are 
too poor to handle topological properties of the shapes to be modelled 
and cannot deal with correlation. Similarly, the belief functions that 
lie at the heart of the Dempster--Shafer theory \cite{Demp67,Shaf76} do 
not necessarily correspond to the containment function of a well-defined 
random closed set \cite{Molc05}.

Here, we propose to combine ideas from pattern analysis
\cite{GrenMill07,Youn10} with the theory of cyclic Gaussian random 
processes to estimate simultaneously the object boundary and the 
noise parameters. In contrast to deformable templates methods
(see e.g. \cite{Bigo11} for a recent example in one dimension), 
in our approach the deformation is not used to model fluctations in 
the appearance of the object of interest but rather to align
parametrisations of the boundary; the fluctuations in appearance 
are taken care of by the noise process. 

The plan of this paper is as follows. In Section~\ref{S:prelim} we recall 
basic facts about planar curves, cyclic Gaussian random processes and
Fourier analysis. In Section~\ref{S:inference} we formulate a model
for sampling noisy curves, carry out inference in the Fourier domain
and quantify the error. Section~\ref{S:align} is devoted to the estimation
of alignment parameters and in Section~\ref{S:applic} we illustrate the
approach on simulated data as well as on a series of observations of
an Ethiopean lake from space. The paper concludes with a discussion and
pointer to future work.

\section{Noisy curves}
\label{S:prelim}

In this section we recall basic facts about planar curves, Fourier
bases and cyclic Gaussian random processes.

\subsection{Planar curves}
\label{S:curves}

Throughout this paper we model the boundary of the random object of 
interest by a smooth (simple) closed curve.

Consider the class of functions 
\(
\Gamma: I \to \oR^2
\)
from some interval $I$ to the plane. Define an equivalence relation
$\sim$ on the function class as follows. Two functions $\Gamma$ and 
$\Gamma^\prime$ are equivalent, $\Gamma \sim \Gamma^\prime$, if there 
exists a strictly increasing function $\varphi$ from $I$ onto another 
interval $I^\prime$ such that $\Gamma = \Gamma^\prime \circ \varphi$. 
Note that $\varphi$ is a homeomorphism. The relation defines a family 
of equivalence classes, each of which is called a {\em curve\/}. Its 
member functions are called {\em parametrisations\/}. Since the images 
of two parametrisations of the same curve are identical, we shall, 
with slight abuse of notation, use the symbol $\Gamma$ for a specific 
parametrisation, for a curve and for its image. 

A curve is said to be continuous if it has a continuous parametrisation, 
in which case all parametrisations are continuous; it is simple if it has 
a parametrisation that is injective. A Jordan curve has the additional 
property of being {\em closed\/}, in other words, it is the image of a 
continuous function $\Gamma$ from $[p, q]$ to $\oR^2$ that is injective 
on $[p, q)$ and for which $\Gamma(p) = \Gamma(q)$. By the 
Jordan--Sch\H{o}nflies theorem, the complement of any Jordan curve in 
the plane consists of exactly two connected components: a bounded one 
and an unbounded one separated by $\Gamma$. The bounded component is
called the interior of $\Gamma$ and can be thought of as the object. 
Since closed curves have neither a `beginning' nor an `end', a 
{\em rooted parametrisation\/} is provided by a point on the curve 
together with a cyclic parametrisation from that point in a given 
direction (say with the interior to the left). For convenience, we 
shall often rescale the definition interval to $[-\pi, \pi]$, 

In the statistical inference to be discussed in the next section, we 
need derivatives. In this context, it is natural to assume a curve to be 
parametrised by some function $\Gamma$ that is $C^1$ and the same degree 
of smoothness to hold for the functions $\varphi$ that define the 
equivalence relation between parametrisations. In effect, $\varphi$ 
should be a diffeomorphism. See \cite[Chapter~1]{Youn10} for further 
details.

\subsection{Fourier representation}
\label{S:Fourier}

Let $\Gamma = (\Gamma_1, \Gamma_2): [-\pi,\pi]\to\oR^2$ be a $C^1$ 
function with $\Gamma_i(-\pi) = \Gamma_i(\pi)$, $i=1, 2$. Recall that 
the family of functions $\{ \cos(jx), \sin(jx) : j\in \oN_0 \}$ forms 
an orthogonal basis for $L_2([-\pi,\pi])$, the space of all square 
integrable functions on $[-\pi,\pi]$, see e.g.\ \cite[Section~12]{GoGo81}, 
so that $\Gamma$ can be approximated by a trigonometric polynomial of the 
form 
\[
 \sum_{j=0}^J \left[ \mu_j \cos (jx) + \nu_j \sin ( jx ) \right].
\]
The vectors $\mu_j$ and $\nu_j$ are called the {\em Fourier\/}
coefficients of order $j$ and satisfy
\begin{equation}
\label{e:FourierGamma}
\left\{ \begin{array}{lll}
\mu_{0,i} & = & \frac{1}{2\pi} \int_{-\pi}^\pi \Gamma_i( \theta) d\theta \\
\mu_{j,i} & = & \frac{1}{\pi} 
  \int_{-\pi}^\pi \Gamma_i( \theta) \cos(j\theta) d\theta \\
\nu_{j,i} & = & \frac{1}{\pi}
  \int_{-\pi}^\pi \Gamma_i( \theta) \sin(j\theta) d\theta 
\end{array} \right.
\end{equation}
for $j\in\oN$ and $i=1, 2$. 
Moreover, by Parseval's identity,
\begin{equation}
\label{e:Parseval}
\frac{1}{\pi} \int_{-\pi}^\pi || \Gamma(\theta) ||^2 d\theta = 
2 || \mu_0 ||^2 + 
\sum_{j=1}^\infty \left[ || \mu_j ||^2 + || \nu_j ||^2 \right].
\end{equation}

\subsection{Stationary cyclic Gaussian processes}
\label{S:noise}

Let $N = ( N_1, N_2 )$ be a stationary cyclic Gaussian process on 
$[-\pi, \pi]$ with values in $\oR^2$ having independent components with 
zero mean and continuous covariance function $\rho$. If the components
$N_i(\theta)$, $i=1,2$, have almost surely continuous sample paths, 
their $j$-th order Fourier coefficients (cf.\ Section~\ref{S:Fourier}) 
are well-defined normally distributed random variables with mean zero 
and variance 
\[
 r_j \int_{-\pi}^{\pi} \rho(\theta) \cos( j\theta) d\theta.
\]
For $j\in \oN$, $r_j = 1/ \pi$, for $j=0$, $r_j = 1/(2\pi)$.
Moreover, all Fourier coefficients are uncorrelated hence independent. 
For details, see e.g.\ \cite[Section~5.3]{CramLead67}. 

Reversely, let $A_{j,i}$ and $B_{j,i}$ be mutually independent 
zero-mean Gaussian random variables with variances $\sigma^2_j$
that are small enough for the series $\sum_j \sigma^2_j$ to converge.
Set, for $\theta \in [-\pi,\pi]$,
\begin{equation}
\label{e:N}
N_{i}(\theta) = \sum_{j=0}^\infty \left[ A_{j,i} \cos( j\theta ) + 
                                    B_{j,i} \sin( j\theta ) \right],
\quad \quad i = 1, 2.
\end{equation}
Then the $N_{i}$ are independent stationary cyclic Gaussian processes 
with zero mean and covariance function 
\[
\rho( \theta) =  \sum_{j=0}^\infty \sigma_j^2 \cos( j\theta ) =
\sigma_0^2 + \sum_{j=1}^\infty \frac{ \sigma_j^2 }{2} \left[ 
  e^{ij\theta} + e^{-ij\theta} \right].
\]
The series is absolutely convergent by assumption. Moreover, $\rho$ is
continuous. However, for the existence of a continuous version, further 
conditions are needed. From the above formula it is clear that the spectral 
measure has density 
\(
m(j) = \sigma_j^2 / 2
\)
on $\oZ \setminus \{ 0 \}$ and $m(0) = \sigma_0^2$. 
Theorem 25.10 in \cite{RogeWill94} then implies that if
\begin{equation}
\label{e:C1}
\sum_{j=1}^\infty j^{2k+\epsilon} \sigma^2_j < \infty
\end{equation}
for $k\in\oN \cup \{ 0 \}$, $\epsilon > 0$, there exists a version of 
$N_{i}$ that is $k$ times continuously differentiable. From now on we 
shall always assume (\ref{e:C1}) for $k=1$.

\begin{ex}
\label{e:Hobolth}
{\rm 
A convenient model is the {\em generalised $p$-order model\/} of
\cite{Hoboetal03}, see also \cite{AletRuff13,JonsJens05},
in which
\[
   \sigma_j^{-2} = \alpha + \beta j^{2p},
\quad \quad j\geq 2,
\]
for parameters $\alpha, \beta > 0$. The parameter $p$ determines
the smoothness. By (\ref{e:C1}), a continuous version exists for 
all $p>1/2$; for $p>3/2$ one that is continuously differentiable. 
}
\end{ex}

\section{Parameter estimation}
\label{S:inference}

\subsection{Data model}

In this paper, the data consist of multiple observations of an object of 
interest in discretised form as a list of finitely many points 
\(
( X^l)_{ l = 1, \dots, n }
\) 
on its boundary, either explicitly (cf.\ Figure~\ref{F:sim}) or implicitly 
in the form of an image as in Figure~\ref{F:Tana}. In other words, the 
lists $(X^l)_l$ trace some unknown closed curve $\Gamma$ affected by noise. 
In the sequel, the number of boundary points, $n$, will be odd.

As discussed in Subsection~\ref{S:curves}, $\Gamma$ may be parametrised by 
a function from $[-\pi, \pi]$ to the plane. As for the noise $N$, in the 
absence of systemetic errors, it is natural to assume that 
$\EE N(\theta) = 0$ for all $\theta\in[-\pi,\pi]$ and that the correlation 
between errors $N(\theta)$ and $N(\eta)$ depends only on the absolute 
difference $|\theta - \eta|$. Thus, we model the noise by independent 
mean-zero stationary cyclic Gaussian processes (\ref{e:N}) on 
$[-\pi, \pi]$.

Alignment between the observed discretised curves is necessary, both to 
fix the roots and to allow for differences in parametrisations. This is 
taken care of by shift parameters $\alpha \in [-\pi, \pi]$ for the 
root and diffeomorphisms $\varphi : [-\pi, \pi] \to [-\pi, \pi]$ for the 
reparametrisation. 

To summarise, we arrive at the following model.

\begin{defn}
\label{d:data}
Let $\Gamma = (\Gamma_1, \Gamma_2): [-\pi,\pi]\to\oR^2$ be a $C^1$ 
function with $\Gamma_i(-\pi) = \Gamma_i(\pi)$, $i=1, 2$. Let
$N_t = ( N_{t,1}, N_{t,2} )$ be independent stationary cyclic Gaussian 
processes on $[-\pi, \pi]$ of the form (\ref{e:N}) with variances
$\sigma^2_j$ for which (\ref{e:C1}) holds. Then, for $\alpha_t \in 
[-\pi, \pi]$ and diffeomorphisms $\varphi_t : [-\pi, \pi] \to [-\pi, \pi]$, 
$\theta_l = - (n+1) \pi/n + 2\pi l / n$, $l=1, \dots, n$, and 
$t = 0, \dots, T$, set
\[
X_t^l = X_t(\theta_l) = \Gamma( \varphi_t( \theta_l - \alpha_t) ) + 
              N_t( \varphi_t( \theta_l - \alpha_t ) ), 
\]
interpreted cyclically modula $2\pi$.
\end{defn}

We set ourselves the goal of estimating $\Gamma$ and the noise variance
parameters $\sigma^2_j$. This is best done in the Fourier domain. For the
moment, assume that all $\alpha_t \equiv 0$ and that each $\varphi_t$ is the 
identity operator.
(We shall return to the issue of estimating these alignment parameters 
in Section~\ref{S:align}). Then Definition~\ref{d:data} reduces to the 
simplified model 
\begin{equation}
\label{e:simple}
X_t(\theta) = \Gamma(\theta) + N_t(\theta),
\end{equation}
which is observed at $X_t^l = X_t(\theta_l)$ 
Under this perfect alignment assumption, $\Gamma$ is a 
$C^1$ rooted parametrisation of the curve of interest with $\Gamma(-\pi) 
= \Gamma(\pi)$. 

It is natural to carry out inference in the Fourier domain. Write 
$\mu_j, \nu_j$ for the Fourier coefficients of $\Gamma$ with components
defined in (\ref{e:FourierGamma}). Let $F_j^t$ and $G_j^t$ be the
random Fourier coefficients of $X_t$ defined by
\begin{equation}
\label{e:Fourier}
\left\{ \begin{array}{lllll}
F_0^t & = & 
\frac{1}{2\pi} \int_{-\pi}^\pi X_t(\theta) d\theta & = & \mu_0 + A_0^t \\
F_j^t & = & \frac{1}{\pi} \int_{-\pi}^\pi X_t(\theta) \cos( j\theta) d\theta 
& = & \mu_j + A_j^t \\ 
G_j^t & = & \frac{1}{\pi} \int_{-\pi}^\pi X_t(\theta) \sin( j\theta) d\theta 
& = & \nu_j + B_j^t
\end{array} \right.
\end{equation}
for $j\in\oN$, where $A_j^t$, $B_j^t$ are as in (\ref{e:N}). Then, the joint 
log likelihood in the Fourier domain of the coefficients up to order 
$J\in\oN$ is 
\begin{eqnarray}
\nonumber
& - & \sum_{t=0}^T \left[
   \log \sigma_0^2  + \sum_{j=1}^J  2 \log \sigma^2_j \right] + \\
\nonumber
& - & \frac{1}{2} \sum_{t=0}^T || f_0^t - \mu_0 ||^2 / \sigma^2_0
-\frac{1}{2} \sum_{t=0}^T \sum_{j=1}^J \left[
|| f_j^t - \mu_j ||^2 + || g_j^t - \nu_j ||^2 \right] / \sigma^2_j 
\end{eqnarray}
upon ignoring constants, where $f_j^t$ and $g_j^t$ are the `observed' 
Fourier coefficients. In practice, one uses a Riemann sum instead of an
integral.

\subsection{Fourier parameter estimation}
\label{S:mle}

In this section we estimate the noise variances $\sigma^2_j$ and the
Fourier coefficients $\mu_j$, $\nu_j$. An estimator for the unknown curve 
$\Gamma$ is obtained by back-transformation.

\begin{res}
\label{t:mle}
The maximum likelihood estimators
\[
\left\{ \begin{array}{lll}
\hat \mu_j & = & \frac{1}{T+1} \sum_{t=0}^T F_j^t, \quad j \in \oN \cup \{0 \}\\
\hat \nu_j & = & \frac{1}{T+1} \sum_{t=0}^T G_j^t, \quad j \in \oN
\end{array} \right.
\]
for the model (\ref{e:simple}) of Definition~\ref{d:data} are mutually 
independent and consistent. They are normally distributed with mean vectors 
$\mu_j$ and $\nu_j$, respectively, and covariance matrix 
\(
\sigma_j^2 I_2 / (T+1)
\), 
writing $I_2$ for the $2\times 2$ identity matrix. For $j\in\oN$, the maximum 
likelihood estimators
\[
\hat \sigma^2_j = \frac{1}{4(T+1)} \sum_{t=0}^T \left[
|| F_j^t - \hat \mu_j ||^2 + || G_j^t - \hat \nu_j ||^2
\right]
\]
are consistent. Moreover, $4(T+1) \hat \sigma_j^2 / \sigma_j^2$ is $\chi^2$ 
distributed with $4 T$ degrees of freedom. The estimator
\[
\hat \sigma^2_0 = \frac{1}{2(T+1)} \sum_{t=0}^T || F_0^t - \hat \mu_0 ||^2
\]
is consistent and $2(T+1) \hat \sigma_0^2 / \sigma_0^2$ is $\chi^2$ 
distributed with $2 T$ degrees of freedom. 
\end{res}

\begin{proof}
The expression for and distribution of the maximum likelihood estimators 
are classic results from multivariate statistics \cite{ChatColl80}. 
The consistency for $T\to\infty$ follows from the law of large numbers
for the mean and the L\'evy--Cram\`er continuity theorem for the variance.

To show independence, fix some finite $J$. Now $F_j^t$ depends only on $A_j^t$,
$G_j^t$ only on $B_j^t$. Hence the random vector consisting of the components
of $F_j^t$, $j\in\{ 0, \dots, J \}$, and $G_j^t$, $j\in \{ 1, \dots, J\}$, for 
all $t=0,\dots, T$ is mutually independent. Since $J$ is arbitrary, the proof 
is complete.
\end{proof}

Transformation to the spatial domain gives an estimator for the unknown 
curve $\Gamma$. Indeed, set
\begin{equation}
\label{e:Gamma}
\hat \Gamma(\theta) = \hat \mu_0 + 
\sum_{j=1}^J \left[ \hat \mu_j \cos(j\theta) + \hat \nu_j \sin(j\theta) \right],
\end{equation}
where $J>0$ is a cut-off value and $\theta \in [-\pi,\pi]$.

\begin{thm}
\label{t:Gamma}
In the model (\ref{e:simple}) of Definition~\ref{d:data}, the estimator 
(\ref{e:Gamma}) is a stationary cyclic Gaussian process with independent 
components. Its mean vector is the Fourier representation
\(
\mu_0 + \sum_{j=1}^J \left[ \mu_j \cos(j\theta) + \nu_j \sin(j\theta) \right]
\) 
of $\Gamma$ truncated at $J$.
The covariance function of both components of (\ref{e:Gamma}) is given by
\(
 \rho_J(\theta) / (T+1)
\)
where $\rho_J(\theta)$ is the truncated covariance function 
$\sum_{j=0}^J \sigma_j^2 \cos( j \theta )$. 
The integrated squared error can be written as
\[
\frac{1}{\pi} \int_{-\pi}^\pi || \hat \Gamma(\theta) - \Gamma(\theta) ||^2 
d\theta =  
 \sum_{j=J+1}^\infty \left[ || \mu_j ||^2 + || \nu_j ||^2 \right]
+ Z_J
\]
where $Z_J = 
 \left( 2\sigma^2_0 Z_0 + \sum_{j=1}^J \sigma^2_j Z_j \right) / (T+1)$ 
and the $Z_j$ are independent $\chi^2$ distributed random variables
with four degrees of freedom for $j\geq 1$ and two for $j=0$.
\end{thm} 

As a simple corollary, the expected integrated squared error is 
\begin{equation}
\label{e:emse}
 \sum_{j=J+1}^\infty \left[ || \mu_j ||^2 + || \nu_j ||^2 \right] +
\frac{4}{T+1} \sum_{j=0}^J \sigma^2_j .
\end{equation}
Note that one has to strike a balance between bias and variance. Indeed,
as $J$ increases, the first term of (\ref{e:emse}) decreases, the second 
one increases. In other words, a decrease in bias leads to an increase 
in variance. Thus, in practice, $J$ has to be chosen carefully, as too 
large a value might result in over-fitting, whereas too small a value 
could lead to over-smoothing.

\

\noindent
\begin{proof}
By Lemma~\ref{t:mle}, (\ref{e:Gamma}) is a Gaussian process with 
independent components and mean function as claimed. Since, by the same 
Lemma, the $\hat \mu_j$ and $\hat \nu_j$ are independent, the covariance 
function of the components $\hat \Gamma_i$, $i=1, 2$, is 
\begin{eqnarray*}
{\rm Cov}\left( \hat\Gamma_i(\theta), \hat \Gamma_i(\eta) \right) & = &
\frac{1}{T+1} \sum_{j=0}^J \sigma_j^2 \left[
  \cos( j\theta)  \cos( j\eta) +  \sin( j\theta)  \sin( j\eta ) 
\right]
\\
& = &
\frac{1}{T+1} \sum_{j=0}^J \sigma_j^2 \cos ( j( \eta - \theta )),
\end{eqnarray*}
a stationary function. By Parseval's identity,
\[
\frac{1}{\pi} \int_{-\pi}^\pi || \hat \Gamma(\theta) - \Gamma(\theta) ||^2 
d\theta =  2 || \hat \mu_0 - \mu_0 ||^2 + 
\sum_{j=1}^\infty \left[ 
  || \hat \mu_j - \mu_j ||^2 +  || \hat \nu_j - \nu_j ||^2 
\right].
\]
The truncation at $J$ of (\ref{e:Gamma}) amounts to setting $\hat\mu_j$ 
and $\hat\nu_j$ to zero for $j>J$. For $j\leq J$,  by Lemma~\ref{t:mle}, 
the components of $\hat \mu_j - \mu_j$ and those of $\hat \nu_j - \nu_j$ 
are independent, normally distributed random variables with variance 
$\sigma^2_j / (T+1)$. Hence, for $j \in \{ 1, \dots, J \}$, 
$|| \hat \mu_j - \mu_j ||^2 +  || \hat \nu_j - \nu_j ||^2 $ divided
by  $\sigma_j^2 / (1+T)$ is $\chi^2$ distributed with four degrees of 
freedom. The random variable $(T + 1) || \hat \mu_0 - \mu_0 ||^2 /
\sigma^2_0$ is $\chi^2$ distributed with two degrees of freedom. 
\end{proof}

To conclude the section, let us turn to asymptotics.

\begin{thm}
Consider the estimator (\ref{e:Gamma}) in the model (\ref{e:simple}) of 
Definition~\ref{d:data}. The integrated squared error 
\[
\frac{1}{\pi} \int_{-\pi}^\pi || \hat \Gamma(\theta) - \Gamma(\theta) ||^2 
d\theta  
\to 
\sum_{j=J+1}^\infty \left[ || \mu_j ||^2 + || \nu_j ||^2 \right]
\]
almost surely as $T\to\infty$.
\end{thm}

It is worth noting that the limit depends solely on the ignored Fourier
coefficients of $\Gamma$.

\

\noindent
\begin{proof}
Recall the notation of Theorem~\ref{t:Gamma}. To prove strong convergence of 
\(
Z_J = Z_J(T)
\)
to $0$ as $T\to\infty$, we use the Borel--Cantelli lemma. Indeed,
\begin{eqnarray*}
\sum_{T=1}^\infty \PP( |Z_J(T) - 0 | \geq \epsilon ) & = & 
\sum_{T=1}^\infty \PP\left( | 2\sigma_0^2 Z_0 + \sum_{j=1}^J \sigma^2_j Z_j | 
\geq (T+1) \epsilon \right) \\
& \leq & 
\sum_{T=1}^\infty \PP(  c_J \chi^2_{4J+2}  \geq (T+1) \epsilon ),
\end{eqnarray*}
where $c_J = \max \{ 2\sigma_0^2, \sigma_1^2, \dots, \sigma_J^2 \}$.
For $T$ large enough, $z(T) = (T+1) \epsilon / c_J  > 4J+2$, and, for
such $T$, the tail probability satisfies
\[
\PP(  \chi^2_{4J+2}  \geq (T+1) \epsilon / c_J ) \leq 
\left( \frac{z(T)}{4J+2} \exp[ 1 - z(T) / (4J+2) ]
\right)^{2J+1}
\]
by the Chernoff bound. Consequently,
\[
\sum_{T=1}^\infty \PP( |Z_J(T) - 0 | \geq \epsilon ) < \infty,
\]
and the strong convergence of $Z_J(T)$ to $0$ follows.
\end{proof}

\subsection{Discretisation}

In practice, the Fourier coefficients (\ref{e:Fourier}) are computed
using a Riemann sum
\begin{equation}
\label{e:Fourier-discrete}
\left\{ \begin{array}{lll}
F_{0,n}^t & = & 
\frac{1}{n} \sum_{l=1}^n X_t^l  
 =  \frac{1}{n} \sum_{l=1}^n \left[ \Gamma(\theta_l) +  N_t(\theta_l) \right] \\
F_{j,n}^t & = & \frac{2}{n} \sum_{l=1}^n X_t^l \cos(j\theta_l) 
 =  \frac{2}{n} \sum_{l=1}^n \left[ \Gamma(\theta_l) \cos(j\theta_l) + 
 N_t(\theta_l) \cos(j\theta_l) \right] \\
G_{j,n}^t & = & \frac{2}{n} \sum_{l=1}^n X_t^l \sin(j\theta_l)
 =  \frac{2}{n} \sum_{l=1}^n \left[ \Gamma(\theta_l) \sin(j\theta_l) + 
 N_t(\theta_l) \sin(j\theta_l) \right] \\
\end{array} \right.
\end{equation}
for $j\in\oN$ and $\theta_l = -(n+1)\pi/n + 2\pi l/n$, $l=1, \dots, n$.
We shall write $\mu_{j,n}$, $\nu_{j,n}$ for the deterministic parts of 
(\ref{e:Fourier-discrete}), $A_{j,n}^t$ and $B_{j,n}^t$ for the stochastic 
ones. As before, $n\geq 3$ is odd. 

In special cases, the Riemann approximation is exact and corresponds to
a discrete Fourier transform. This is the content of the next result.
Its proof will be used later on in this section.

\begin{res}
\label{DFT}
Suppose that the Fourier transforms of $\Gamma$ and $N_t$ vanish from
order $J+1$ onwards where $J \leq (n-1)/2$, $n$ odd, and 
$\theta_l = - (n+1) \pi/n + 2\pi l/n$, 
$l=1, \dots, n$. Then, for $j\in \{0, \dots, J\}$, $F_j^t = \mu_j + A_j^t$
and $G_j^t = \nu_j + B_j^t$.
\end{res}

\begin{proof} Recall the Lagrange identities. For $\alpha \in (0,2\pi)$,
\begin{eqnarray*}
\sum_{l=1}^n \sin(l\alpha) & = & \frac{1}{2} \cot(\frac{\alpha}{2})
- \frac{ \cos( (n+\frac{1}{2}) \alpha )}{2 \sin(\frac{\alpha}{2})}; \\
\sum_{l=1}^n \cos(l\alpha) & = & - \frac{1}{2}
+ \frac{\sin( (n+\frac{1}{2}) \alpha )}{2 \sin(\frac{\alpha}{2})} .
\end{eqnarray*}

First note that 
\(
\exp\{ i j \theta_l \} 
\)
$l=1, \dots, n$ and $j = -(n-1)/2 , \dots, (n-1)/2$ is an orthogonal family.
To see this, take $j_1$, $j_2$ and compute the inner product
\[
\sum_{l=1}^n e^{i j_1 \theta_l} e^{-i j_2 \theta_l} = 
\sum_{l=1}^n e^{i (j_1-j_2) \theta_l} = 
\sum_{l=1}^n \left[ \cos( (j_1 - j_2)\theta_l ) + i \sin( (j_1-j_2)\theta_l) 
\right] .
\]
Since $(j_1 - j_2) \theta_l = - (j_1 - j_2) \pi (n+1)/n + 2 \pi (j_1-j_2) l/n$,
we may use the Lagrange identities with, for $j_1 > j_2$,
$\alpha = 2 \pi (j_1-j_2) / n$ provided $\alpha \in (0, 2\pi)$, that is, 
$j_1 \neq j_2$ and $|j_1-j_2| < n$.
The latter is true by assumption. Writing $j\neq 0$ for $|j_1-j_2|$ we get
\begin{eqnarray*}
\sum_{l=1}^n \sin( j\theta_l) & = & 
 \sin( - \frac{ n + 1}{n} j \pi ) \sum_{l=1}^n \cos( l \frac{ 2 \pi j }{n} ) + 
 \cos( - \frac{ n + 1}{n} j \pi ) \sum_{l=1}^n \sin( l \frac{ 2 \pi j }{n} ) = 0; 
\\
\sum_{l=1}^n \cos( j\theta_l) & = & 
 \cos( - \frac{ n + 1}{n} j \pi ) \sum_{l=1}^n \cos( l \frac{ 2 \pi j }{n} ) - 
 \sin( - \frac{ n + 1}{n} j \pi ) \sum_{l=1}^n \sin( l \frac{ 2 \pi j }{n} ) = 0.
\end{eqnarray*}
When $j=0$, that is $j_1 = j_2$, clearly $\sum_l \cos(j\theta_l) = n$ and
$\sum_l \sin(j\theta_l) = 0$. For negative $j$, analogous computations can
be done so that the orthogonality proof is complete. 

To conclude the proof, use the identities $\cos x = ( e^{ix} + e^{-ix} ) / 2$
and $\sin x = ( e^{ix} - e^{-ix} ) / 2i$ to derive that for 
$j_1, j_2 \in \{ 1, \dots, J \}$,
\begin{equation}
\label{e:basis}
\sum_{l=1}^n \cos(j_1 \theta_l) \cos( j_2 \theta_l)  =  
\sum_{l=1}^n \sin(j_1 \theta_l) \sin( j_2 \theta_l)  =  
\frac{n}{2} \, {\bf 1} \{ j_1 = j_2  \} 
\end{equation}
and
\(
\sum_{l=1}^n \cos(j_1 \theta_l) \sin( j_2 \theta_l)  =  0.
\)
\end{proof}


To estimate $\Gamma$, transform back from the Fourier to the spatial domain. 
Again, we assume $J < n/2$ to make sure that the number of Fourier parameters 
to estimate is not greater than the number of observed boundary points. 
Indeed, set
\begin{eqnarray}
\nonumber
\widehat{\Gamma_n}(\theta) & = & 
\frac{1}{(T+1) n} \sum_{t=0}^T \sum_{l=1}^n X_t^l
\\
\nonumber
& + &
\frac{2}{(T+1)n} \sum_{t=0}^T \sum_{l=1}^n X_t^l \sum_{j=1}^J \left[
\cos(j\theta_l) \cos(j\theta) + \sin(j\theta_l) \sin(j\theta) 
\right] \\
& = & 
\label{e:smooth}
\frac{1}{(T+1)} \sum_{t=0}^T \sum_{l=1}^n X_t^l  \left[
\frac{1}{n} + \frac{2}{n} \sum_{j=1}^J \cos( j(\theta - \theta_l) )
\right].
\end{eqnarray}
We shall use the notation 
$S_l(\theta) = 1/n + 
2 \sum_{j=1}^J \cos( j(\theta - \theta_l) )/n$ for the `smoothing'.

\begin{thm}
\label{t:discrete}
The estimator (\ref{e:smooth}) in model (\ref{e:simple}) is a stationary 
cyclic Gaussian process with independent components. Its mean vector is 
the Riemann approximation to the Fourier representation
\[
\sum_{l=1}^n \Gamma(\theta_l) S_l(\theta) =
\mu_{0,n} + \sum_{j=1}^J \left[
  \mu_{j,n} \cos(j\theta) + \nu_{j,n} \sin(j\theta) \right]
\]
of $\Gamma$ truncated at $J$. 
Provided $J < n/2$ the covariance function of both components of
(\ref{e:smooth}) is given by
\(
 \rho_{J,n}(\theta) / (T+1)
\)
where $\rho_{J,n}$ is the truncated covariance function
\(
\sum_{j=0}^J \sigma^2_{j,n} \cos( j \theta  )
\)
based on the Riemann approximations 
$\sigma^2_{j,n} = 2 \sum_l \rho(\theta_l) \cos(j \theta_l) / n$
for $j\geq 1$ and $\sigma^2_{0,n} =\sum_l \rho(\theta_l) / n$. 
\end{thm}

\begin{proof}
It follows immediately from Definition~\ref{d:data} that, for each 
$t = 0, \dots, T$, the random vector $X_t = (X_t^1, \dots, X_t^n)$ 
is normally distributed. Its mean vector consists of the $\Gamma(\theta_l)$.
Its components are independent, and the covariance matrix $\Sigma$ of each has
entries $\Sigma_{lm} = \rho(\theta_m - \theta_l)$. Moreover, the random vectors
$X_t$ are independent. Therefore, 
\[
\EE \widehat{\Gamma_n}(\theta) = \sum_{l=1}^n \Gamma(\theta_l) S_l(\theta)
\]
is as claimed upon using the classic trigonometric formula for the 
cosine of a sum. Also,
\begin{equation}
\label{e:Cov}
{\rm Cov}(\widehat{\Gamma_{n,i}}(\theta), \widehat{\Gamma_{n,i}}(\eta)) = 
\frac{1}{T+1} \sum_{l=1}^n \sum_{m=1}^n \rho(\theta_m-\theta_l) 
S_l(\theta) S_m(\eta)
\end{equation}
for $i=1,2$; different components are independent.
Now
\begin{eqnarray*}
\sum_{l=1}^n \rho(\theta_m - \theta_l) S_l(\theta) & = &
\frac{1}{n} \sum_{l=1}^n \rho(\theta_m - \theta_l) +
\frac{2}{n} \sum_{l=1}^n \sum_{j=1}^J \rho(\theta_m - \theta_l) 
\cos(j (\theta - \theta_m) + j (\theta_m - \theta_l) ) \\
& = & 
\sigma^2_{0,n} + \sum_{j=1}^n \sigma^2_{j,n} \cos( j (\theta - \theta_m) ) - 0
\end{eqnarray*}
by the trigonometric formula for the cosine of a sum, the fact that, 
for fixed $m$,  $\theta_m - \theta_l$ cyclically interpreted run through 
the same values as $\theta_l$, and the anti-symmetry of the sine function. 
Consequently, (\ref{e:Cov}) reads
\[
\frac{1}{T+1} \sum_{m=1}^n \left[ \frac{1}{n} + \frac{2}{n} \sum_{i=1}^J
\cos( i(\eta - \theta_m) ) \right]
\times
\left[
\sigma^2_{0,n} + \sum_{j=1}^J \sigma^2_{j,n} \cos( j (\theta - \theta_m) ) 
\right]. 
\]
To conclude the proof, note that, by the proof of Lemma~\ref{DFT},
\begin{eqnarray*}
\sum_{m=1}^n \cos( i( \eta-\theta_m) ) \cos( j( \theta-\theta_m) ) & = &
\left[ \cos( j \eta) \cos( j\theta ) + \sin(j\eta) \sin(j\theta) \right]
 c_j  {\bf 1} \{ i = j \} \\
& = &
c_j \cos( j( \eta - \theta))  {\bf 1} \{ i = j \} 
\end{eqnarray*}
for $i, j= 0, \dots, J$, where $c_0 = n$ and $c_j = n/2$ for $j\geq 1$.
\end{proof}

\begin{thm}
Consider the estimator (\ref{e:smooth}) in the model (\ref{e:simple}) 
of Definition~\ref{d:data} and assume $J < n/2$. Then
\[
\frac{1}{\pi}
\int_{-\pi}^\pi || \widehat{\Gamma_n}(\theta) - \Gamma(\theta) ||^2 d\theta 
= \sum_{j=J+1}^\infty \left[ || \mu_j ||^2 + || \nu_j ||^2 \right] +
 Z_{J,n}
\]
where $Z_{J,n} = \left( 
2 \sigma^2_{0,n} Z_0 + \sum_{j=1}^J \sigma^2_{j,n} Z_{j,n} \right) / (T+1) $
and the $Z_{j,n}$ are independent $\chi^2$ distributed random variables with
four degrees of freedom for $j\geq 1$, two for $j=0$, and non-centrality
parameters $(T+1) c_{j,n} / \sigma^2_{j,n}$ with
\[
c_{j,n} = || \mu_{j,n} -\mu_j ||^2  + || \nu_{j,n} -\nu_j ||^2
\]
for $j=1, \dots, J$ and $c_{0,n} = || \mu_{0,n} -\mu_0 ||^2$ for $j=0$.
Moreover,
\(
Z_{J,n} \to  2 c_{0,n} + \sum_{j=1}^J c_{j,n}
\)
almost surely as $T\to\infty$.
\end{thm}

Note that the expected integrated squared error compared to 
(\ref{e:emse}) gains a factor $( 2 c_{0,n} + \sum_{j=1}^J c_{j,n})$ 
due to discretisation errors, except in the special case of Lemma~\ref{DFT}.

\

\noindent
\begin{proof}
By Parseval's identity
\[
\frac{1}{\pi} 
\int_{-\pi}^\pi || \widehat{\Gamma_n}(\theta) - \Gamma(\theta) ||^2
d\theta = 2 || \hat \mu_{0,n} - \mu_0 ||^2 +
\sum_{j=1}^\infty \left[ || \hat \mu_{j,n} - \mu_j ||^2 +
 || \hat \nu_{j,n} - \nu_j ||^2 \right]
\]
where $\hat\mu_{j,n}$ and $\hat \nu_{j,n}$ are the Fourier coefficients
of (\ref{e:smooth}). Due to the truncation of (\ref{e:smooth}) at $J$,
$\hat\mu_{j,n} = \hat\nu_{j,n} = 0$ for $j\geq J+1$.

Note that $\hat \mu_{j,n} - \mu_j$ and $\hat\nu_{j,n} - \nu_j$ are normally 
distributed with mean vectors $\mu_{j,n} - \mu_j$ and $\nu_{j,n} - \nu_j$, 
respectively. The covariance matrices are diagonal with entries 
$\sigma^2_{j,n}/(T+1)$. For $j=0$, this follows by direct computation 
upon recalling that, for fixed $l$, $\theta_m - \theta_l$ interpreted
cyclically run through the same values as $\theta_m$. For $j=1, \dots, J$,
the covariance entry is
\[
\frac{1}{T+1} \, \frac{4}{n^2} \sum_{l=1}^n \cos(j\theta_l) \sum_{m=1}^n
\rho(\theta_m - \theta_l) \cos( j(\theta_m - \theta_l) + j\theta_l).
\]
By the trigonometric formula for the cosine of a sum, the anti-symmetry
of the sine function and the observation that $\sum_l \cos^2(j\theta_l)
= n / 2$ under the given assumptions, we conclude that the covariance
entry is equal to $\sigma^2_{j,n}/(T+1)$. A similar reasoning applies to
$\hat \nu_{j,n}$.

To see that the family consisting of $\hat \mu_{j,n}$ for $j=0, \dots, J$
and $\hat\nu_{j,n}$ for $j=1, \dots, J$ is uncorrelated (hence independent),
once again use (\ref{e:basis}) in combination with the orthogonality of 
$\cos(j_1\theta_l)$ and $\sin(j_1\theta_l)$. The Lagrange identities imply 
that ${\rm Cov}(\hat\mu_{0,n}, \hat\mu_{j,n})= 0$ and 
${\rm Cov}(\hat\mu_{0,n}, \hat\nu_{j,n}) = 0$. 

We conclude that, for $j=1, \dots, J$,
$|| \hat \mu_{j,n} - \mu_j ||^2 +  || \hat \nu_{j,n} - \nu_j ||^2 $ 
multiplied by $(T+1)/\sigma^2_{j,n}$ is the sum of four independent
squared normals with different means, that is, a non-central $\chi^2$
distributed random variable with four degrees of freedom and 
non-centrality parameter $(T+1) c_{j,n} / \sigma^2_{j,n}$ with
$c_{j,n} = || \mu_{j,n}-\mu_j||^2 + || \nu_{j,n} - \nu_j||^2$. 
For $j=0$, 
$|| \hat \mu_{j,n} - \mu_j ||^2 $ multiplied by $(T+1)/\sigma^2_{0,n}$
is the sum of two squared normals, hence a non-central $\chi^2$
distributed random variable with two degrees of freedom and non-centrality
parameter $(T+1) c_{0,n} / \sigma^2_{0,n}$ such that
$c_{0,n} = || \mu_{0,n}-\mu_0||^2 $.

Turning to asymptotics, since the components of $F_j^{t}$ have finite 
variance, Kolmogorov's strong law of large numbers implies almost sure 
covergence of $\hat \mu_{j,n} - \mu_j$ to $\mu_{j,n} - \mu_j$. 
The same holds for the $G_j^t$. Therefore $Z_{J,n}$ converges strongly to 
$2 c_{0,n} + \sum_{j=1}^J c_{j,n}$.
\end{proof}


\section{Alignment}
\label{S:align}

Most data do not come in perfectly registered form and need to be aligned.
Section~\ref{S:diffeo} discusses how diffeomorphisms can be used for
this purpose; Section~\ref{S:register} derives estimators for the alignment 
parameters.

\subsection{Diffeomorphisms}
\label{S:diffeo}

Recall that, given a root, any parametrisation $\Gamma$ of a (simple) 
closed $C^1$ curve can be written as a composition $\Gamma^\prime \circ \varphi$ 
of a fixed parametrisation $\Gamma^\prime$ (say the arc length from the root) 
with a diffeomorphism $\varphi$, cf.\ Section~\ref{S:curves}. Thus, given two 
curves parametrised by, say, $\Gamma$ and $\Gamma_1$, alignment of $\Gamma_1$ 
to $\Gamma$ amounts to finding a shift $\alpha$ to get a common beginning and 
a diffeomorphism $\varphi$ to move along the curve at equal speed such that
$\Gamma_1(\theta) \approx \Gamma( \varphi(\theta - \alpha) )$ interpreted 
cyclically. Without loss of generality, we consider diffeomorphisms 
$\varphi$ from $[-\pi,\pi]$ onto itself.

Parametric diffeomorphisms can be constructed as the flow of differential 
equations \cite[Chapter~8]{Youn10}. In our context, it is convenient to 
consider the differential equation
\begin{equation}
\label{e:ode}
x^\prime(t) = f_w(x(t)), \quad \quad t \in \oR,
\end{equation}
with initial condition $x(0) = \theta \in [-\pi, \pi]$.
Heuristically, consider a particle whose position at time $0$ is $\theta$.
If the particle travels with speed governed by the function $f_w$, then $x(t)$ 
is its position at time $t$. To emphasise the dependence on the initial 
state we shall also write $x_\theta(t)$. 

We let $f_w$ be a trigonometric polynomial, that is, a
linear combination of Fourier basis functions with pre-specified
values $w_i$ at equidistant $x_i \in [-\pi, \pi]$ under the constraint
that $f_w(-\pi) = f_w(\pi) = 0$. More precisely, 
let $-\pi = x_0 < x_1 < \cdots < x_{2m} < \pi$, $w_0 = 0$, and define
\[
f_w(x) = \sum_{j=0}^{2m} w_j t_j(x)
\]
where
\begin{equation}
\label{e:tripo}
t_j(x) = \frac{ \prod_{j \neq k=0}^{2m} \sin\left( \frac{x - x_k}{2} \right) }
{ \prod_{j \neq k=0}^{2m} \sin\left( \frac{x_j - x_k}{2} \right) }
\end{equation}
for arbitrary $w_1, \dots, w_{2m}$ and $m\geq 1$. By \cite[Theorem~8.7]{Youn10},
the function
\[
\theta\mapsto x_\theta(1) = 
  \theta + \int_0^1 \sum_{j=0}^{2m} w_j t_j( x_\theta(t) ) dt,
\]
the solution of (\ref{e:ode}) at time $1$, is a diffeomorphism of 
$[-\pi, \pi]$. This function is known as the {\em flow\/} of the differential
equation and denoted by $\varphi(\theta) = x_\theta(1)$. Since the flow 
depends on the weights, we shall also write $\varphi_w(\theta)$ to emphasise 
this fact. In the next section, we shall need the derivative of
(\ref{e:tripo}), which is given by
\[
t_j^\prime(x) = \frac{ \sum_{j \neq k=0}^{2m} \cos\left( \frac{x - x_k}{2} \right)
\prod_{j,k \neq i=0}^{2m}
\sin\left( \frac{x - x_i}{2} \right)
}
{ 2 \prod_{j \neq k=0}^{2m} \sin\left( \frac{x_j - x_k}{2} \right) }.
\]

Note that in total, there are $2m + 1$ alignment parameters, $2m$ 
for the diffeomorphism and one for the shift in starting point.

\subsection{Inference on alignment parameters}
\label{S:register}

Return to the model introduced in Definition~\ref{d:data}, that is,
\[
X_t(\theta) = \Gamma( \varphi_{w_t} (\theta - \alpha_t)  ) +
              N_t( \varphi_{w_t} (\theta - \alpha_t) ) 
\]
observed at $\theta_l = -(n+1)\pi/n + 2\pi l/n$, $l=1, \dots, n$, and
extended to $[-\pi,\pi]$ by trigonometric interpolation. The latter
is valid, since $n$ is odd. By (\ref{e:smooth}), 
\(
\widehat{\Gamma_n}(\theta) =  \sum_{t=0}^T \hat \Gamma_t(\theta) / (T+1)
\)
where 
\[
\hat \Gamma_t(\theta)  =
\sum_{l=1}^n X_t( \varphi_{-w_t}(\theta_l) +\alpha_t ) S_l(\theta)
\]
is a smoother for the $t$-th curve.
Therefore, the alignment parameters may be estimated by minimising
\begin{equation}
\label{e:M}
M_n(\alpha_0, \dots, \alpha_t, w_0, \dots, w_t) = \sum_{t=0}^T \sum_{l=1}^n
|| \hat \Gamma_t(\theta_l) -  \widehat{\Gamma_n}(\theta_l) ||^2,
\end{equation}
the Riemann sum approximation to the total $L_2$-distance between 
the smoothed data curves and the estimated `true' curve after alignment. 

Without constraints, (\ref{e:M}) is unidentifiable. To see this, note that
for any diffeomorphism $\varphi$ and any shift $\alpha$, 
\[
\int_{-\pi}^\pi || \Gamma( \varphi_t(\theta - \alpha_t) ) - 
\frac{1}{T} \sum_{t=0}^T \Gamma( \varphi_t(\theta - \alpha_t) ) ||^2 d\theta 
\]
is zero whenever $\alpha_t \equiv \alpha$ and $\varphi_t \equiv \varphi$.
We shall use the constraint $\alpha_0 =0$ for the root point. For the 
weight vector, one may set $w_0 = 0$ corresponding to the identity map.
If the points of $X_0^l$ do not cover the curve well, an alternative
is to constrain the average $\sum_t w_t$ to zero.

To optimise $M_n$ over its arguments, one needs its derivatives.

\begin{res} 
Consider the model of Definition~\ref{d:data} and use trigonometric
interpolation for $X(\cdot)$. Then, the partial derivatives of (\ref{e:M})
are, for $t=0, \dots, T$ an $i=1, \dots, 2m$,
\begin{eqnarray*}
\frac{\partial M}{\partial \alpha_{t}} & = &
2 \sum_{l=1}^n \left[
\hat \Gamma_{t}(\theta_l) -\widehat{\Gamma_n}(\theta_l)
\right]^T
\sum_{k=1}^n S_k(\theta_l) 
   X_t^\prime( \varphi_{-w_t}(\theta_k) + \alpha_t ); \\
\frac{\partial M}{\partial w_{t, i}}  & = &
2 \sum_{l=1}^n \left[
\hat \Gamma_{t}(\theta_l) -\widehat{\Gamma_n}(\theta_l)
\right]^T
\sum_{k=1}^n S_k(\theta_l) 
\frac{\partial}{\partial w_{t, i}} \varphi_{-w_t}(\theta_k)
   X_t^\prime( \varphi_{-w_t}(\theta_k) + \alpha_t ).
\end{eqnarray*}
\end{res}

\begin{proof}
Write $z_t$ for a generic component of the alignment parameter of curve
$t=1, \dots, T$. Then
\begin{eqnarray*}
\frac{\partial M}{\partial z_{t}}  & = &
\sum_{s=0}^T \sum_{l=1}^n 2 \left( \hat \Gamma_s(\theta_l) -
\widehat{\Gamma_n}(\theta_l) \right)^T \left[
1\{ t = s \} \frac{\partial}{\partial z_t} 
\hat \Gamma_s(\theta_l) 
- \frac{1}{T+1} \frac{\partial}{\partial z_t}
\hat \Gamma_t(\theta_l) 
\right] \\
& = &
2 \sum_{l=1}^n \left[
\hat \Gamma_{t}(\theta_l) -\widehat{\Gamma_n}(\theta_l)
\right]^T
\frac{\partial}{\partial z_t} \hat \Gamma_t( \theta_l) .
\end{eqnarray*}
Now
\[
\frac{\partial}{\partial z_t}  \hat \Gamma_{t}(\theta_l) =
\sum_{k=1}^n S_k(\theta_l) \frac{\partial}{\partial z_t }
 X_t( \varphi_{-w_t}(\theta_k) + \alpha_t),
\]
from which the claim follows by the chain rule.
\end{proof}

It is well-known from the theory of ordinary differential equations 
\cite[Chapter~1.7]{CoddLevi55} that the partial derivative of 
\(
\varphi_{-w_{t,i}}(\theta)
\)
with respect to $w_{t,i}$ 
is the unique solution of the differential equation
\[
\frac{\partial }{\partial s}  u(s)   
= f^\prime_{-w_t} ( x_\theta(s) )  u(s) - t_i( x_\theta(s) )
\]
at time $s=1$ with initial value $u(0) = 0$
where $x_\theta(s)$ is a solution of (\ref{e:ode}) with weight 
vector $w = -w_t$. 

Having estimated the alignments, the theory of Section~\ref{S:inference}
may be applied to the transformed contours 
$Y_t(\theta) = X_t(\varphi_{-\hat w_t}(\theta + \hat \alpha_t) )$.

\section{Applications}
\label{S:applic}

In this section, we apply the techniques discussed in
Section~\ref{S:inference}--\ref{S:align} to simulated and real life data. 
We work in {\tt R} and use the R-package {\tt deSolve} \cite{Soetetal10}
for solving the differential equations involved.

\begin{figure}[thb]
\begin{center}
\centerline{
\epsfxsize=0.3\hsize
\epsffile{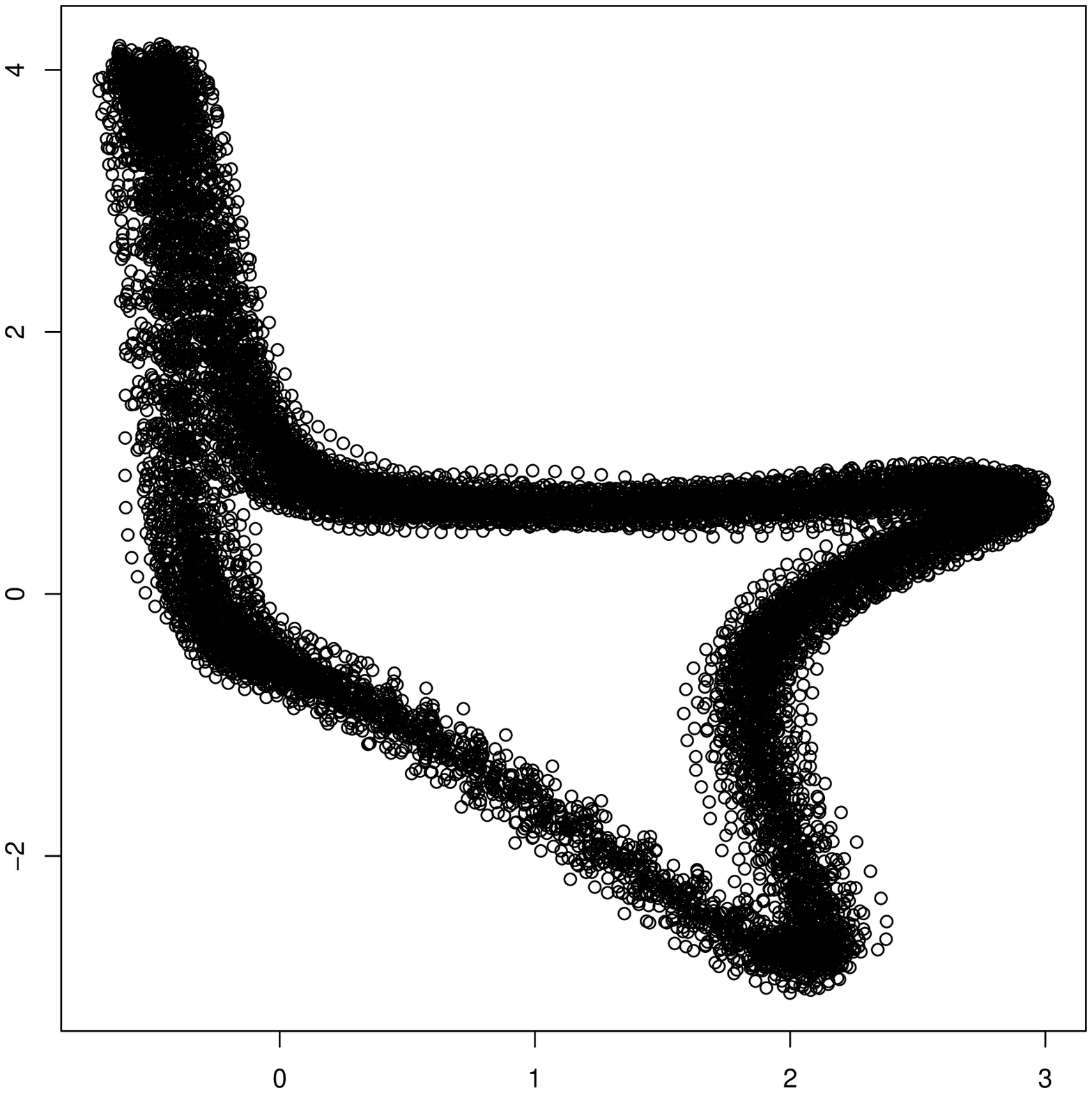}
\epsfxsize=0.3\hsize
\epsffile{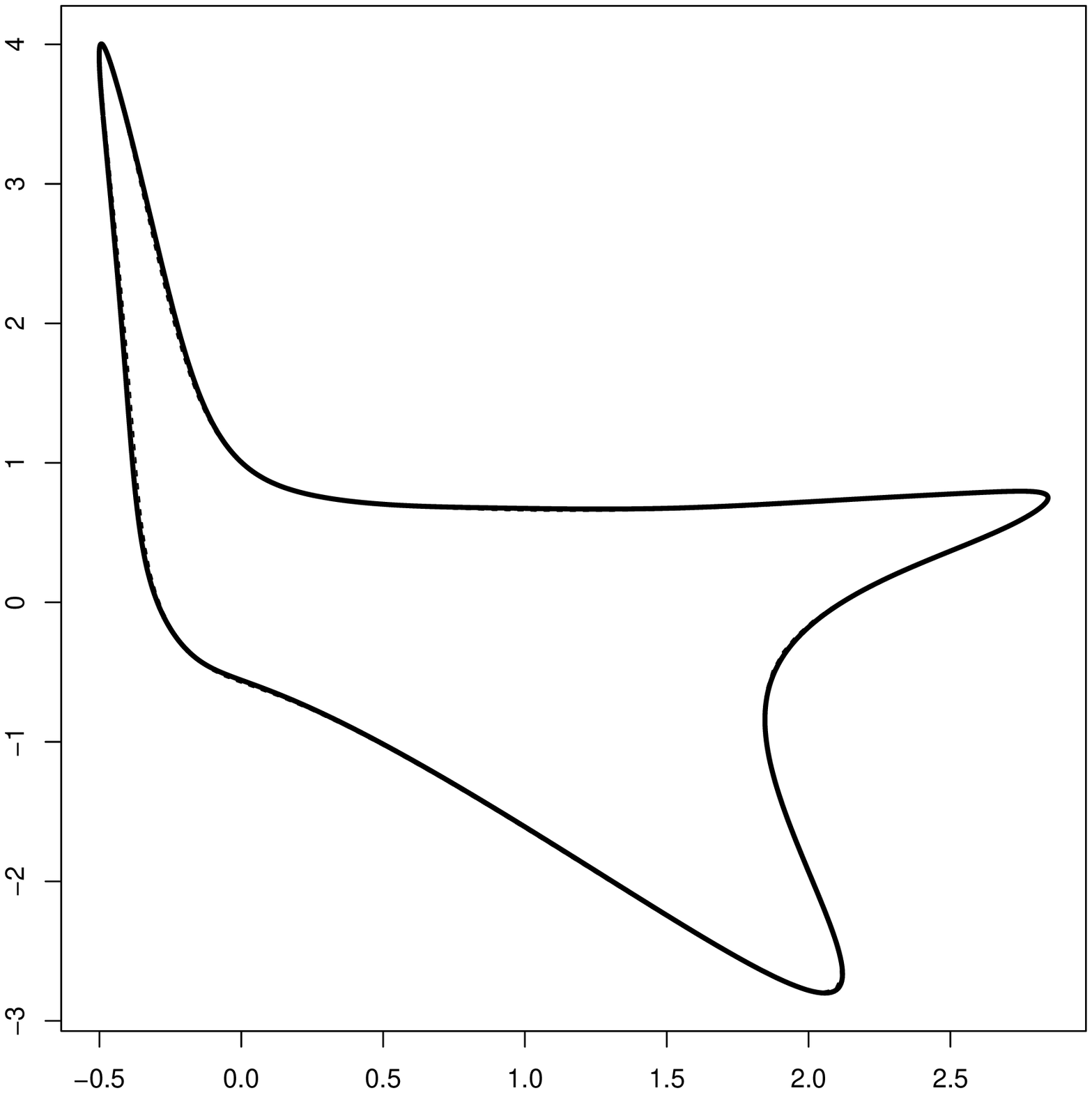}
\epsfxsize=0.3\hsize
\epsffile{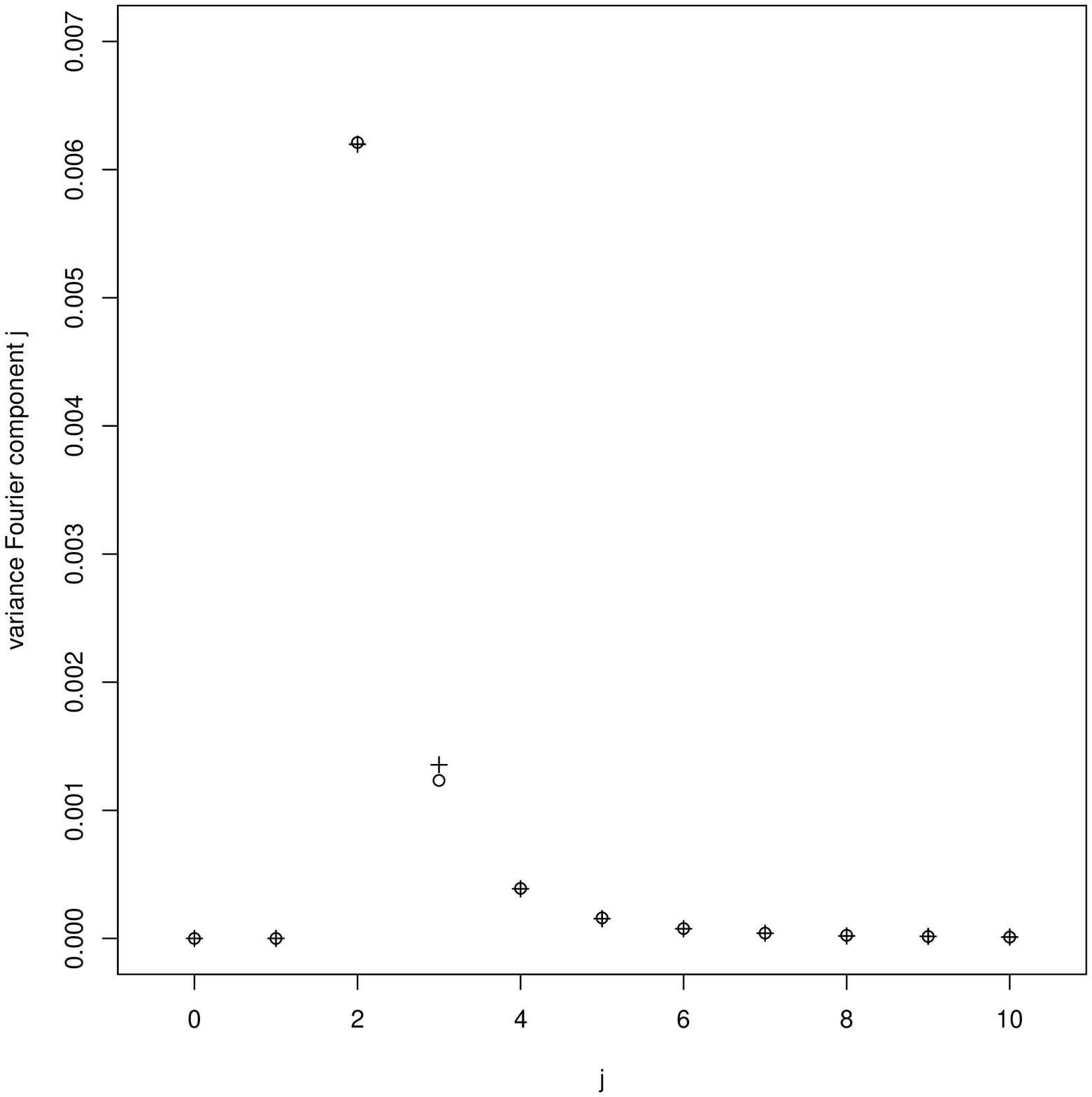} }
\end{center}
\caption{Left-most panel: Data points sampled along $100$ curves. 
Middle panel: Estimated (solid line) and true curve (dashed line). 
Right-most panel: Estimated variance $\hat \sigma^2_j$ of the Fourier 
coefficients plotted against $j$ (crosses) compared to their true 
value $\sigma_j^2$ (circles).}
\label{F:sim}
\end{figure}

\subsection{Simulated example}


The left-hand panel in Figure~\ref{F:sim} shows a hundred contours 
consisting of points sampled at $\theta_l = -\pi + l/20$, 
$l=0, \dots, 125$, along a nested quintic curve, 
cf.~\cite{Kere04}, degraded by noise.
For the noise we use the generalised $p$-order model \cite{Hoboetal03} 
discussed in Example~\ref{e:Hobolth} with $p=2$, $\alpha = 1.0$ and 
$\beta = 10.0$, truncated at ten Fourier coefficients. Note that the
sample paths are almost surely continuously differentiable.

We use equation (\ref{e:Gamma}) to estimate the true curve $\Gamma$. The 
result is shown as the solid line in the middle panel of Figure~\ref{F:sim}.
The truth is shown as the dashed line in the same panel. It can be seen that
the match is excellent. 

We also estimate the variances $\sigma^2_j$ for $j=0, \dots, 10$, according 
to Lemma~\ref{t:mle}. These are shown as crosses in the right-most panel of 
Figure~\ref{F:sim}. For comparison, the true values are plotted too (the 
circles in the right-most panel of Figure~\ref{F:sim}). 

\subsection{Lake Tana}


Figure~\ref{F:Tana} shows three images of Laka Tana, the largest lake in 
Ethiopia and the source of the Blue Nile. It is located near the centre of 
the high Ethiopian plateau and covers some 1400 square miles. Clearly 
visible is Dek island, site of historic monasteries, in the south-central 
portion of the lake, which we shall use as the centre of our coordinate system.

The three images were downloaded from NASA's `The Gateway to Astronaut 
Photography of Earth' website
\begin{verbatim}
http://eol.jsc.nasa.gov/scripts/sseop/photo.pl?mission=STS098&roll=711&frame
\end{verbatim}
(frames $23, 24, 25$). The images were taken on February 17th, 2001, at one 
second intervals by astronauts on the STS098 mission from a space craft 
altitude of $383$ km. 
The centre is at latitude 12.0 and longitude 37.5 degrees. The cloud cover 
is about $25\%$.

\begin{figure}[thb]
\begin{center}
\centerline{
\epsfxsize=0.3\hsize
\epsffile{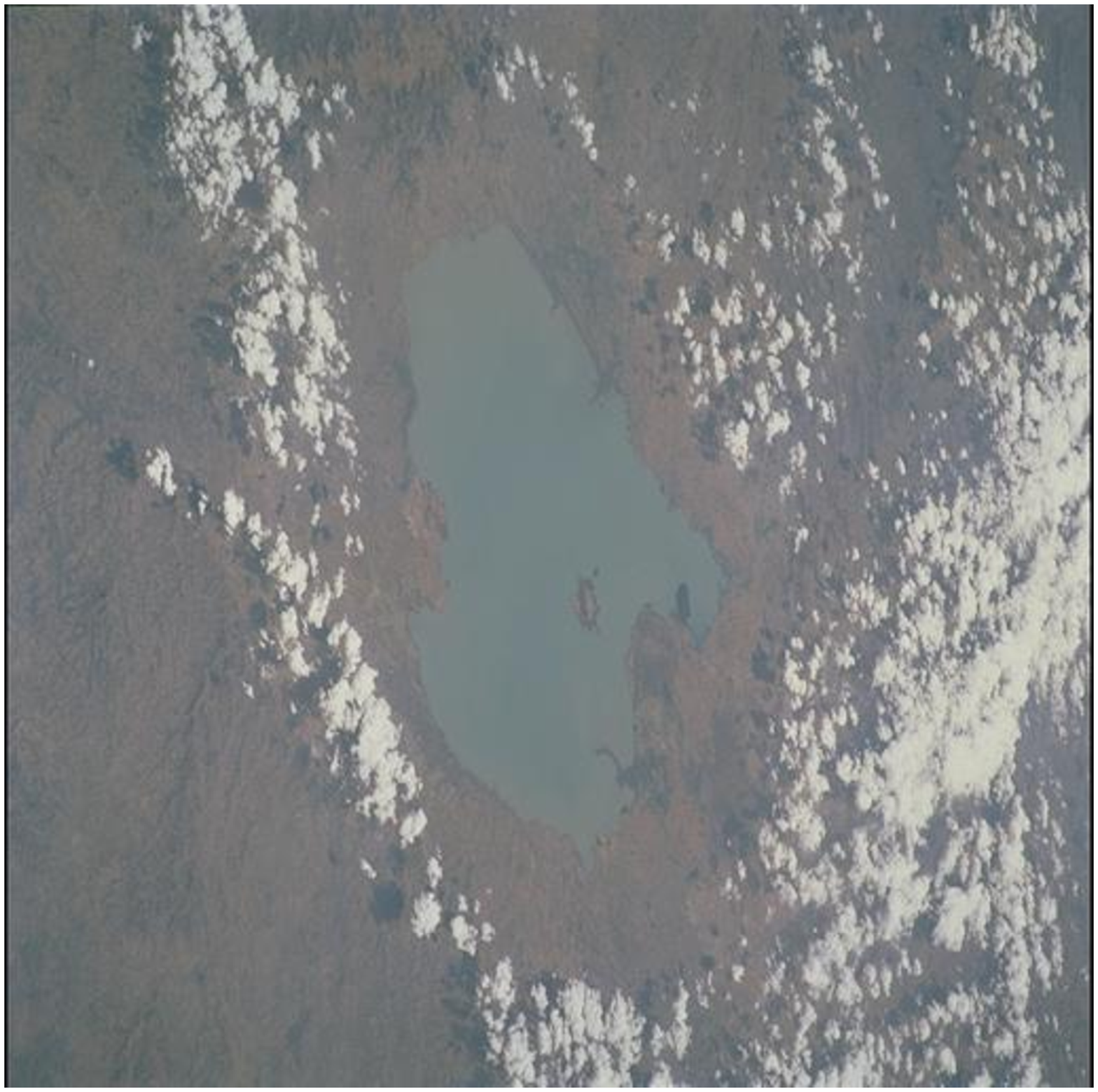}
\epsfxsize=0.3\hsize
\epsffile{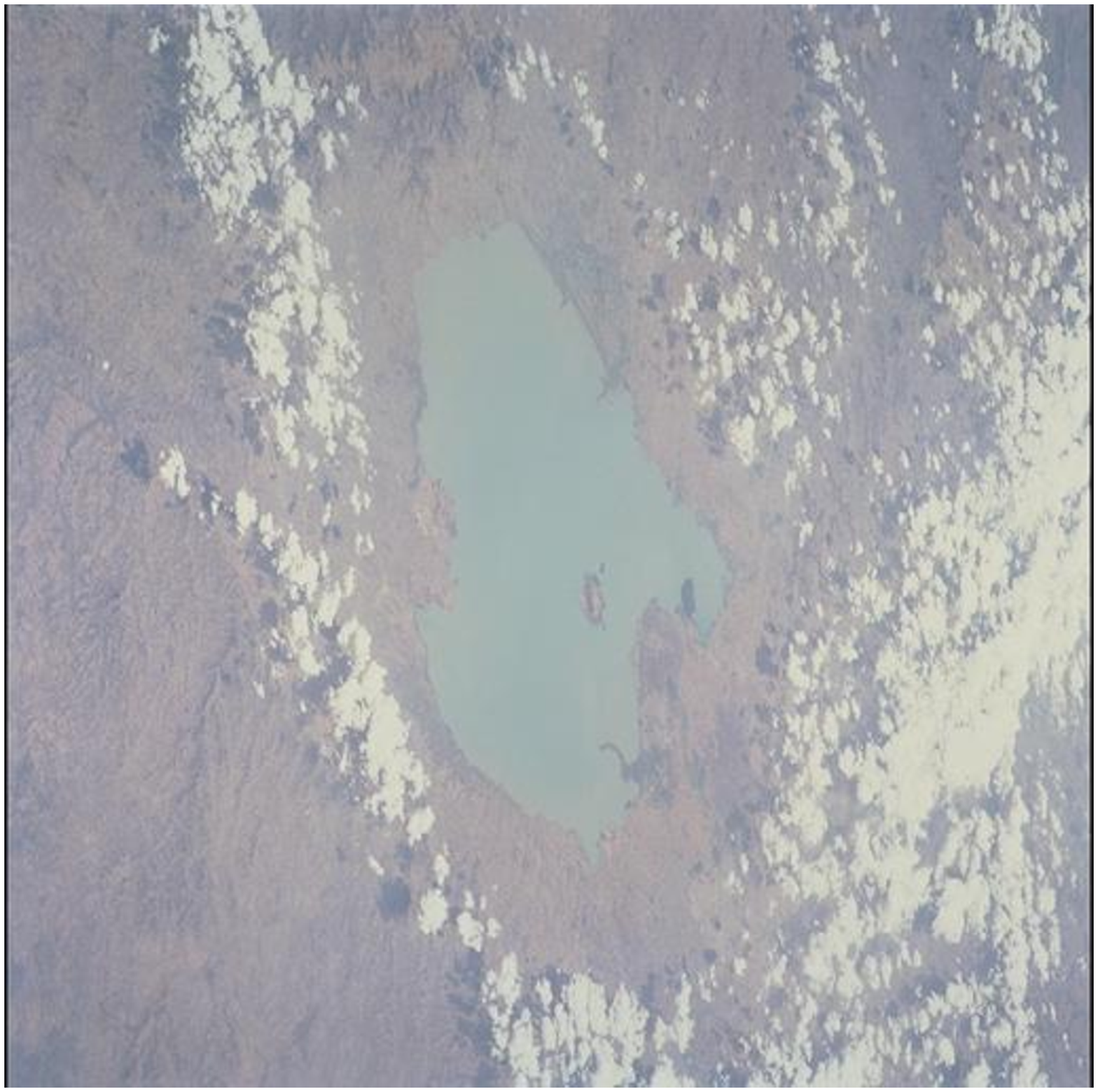}
\epsfxsize=0.3\hsize
\epsffile{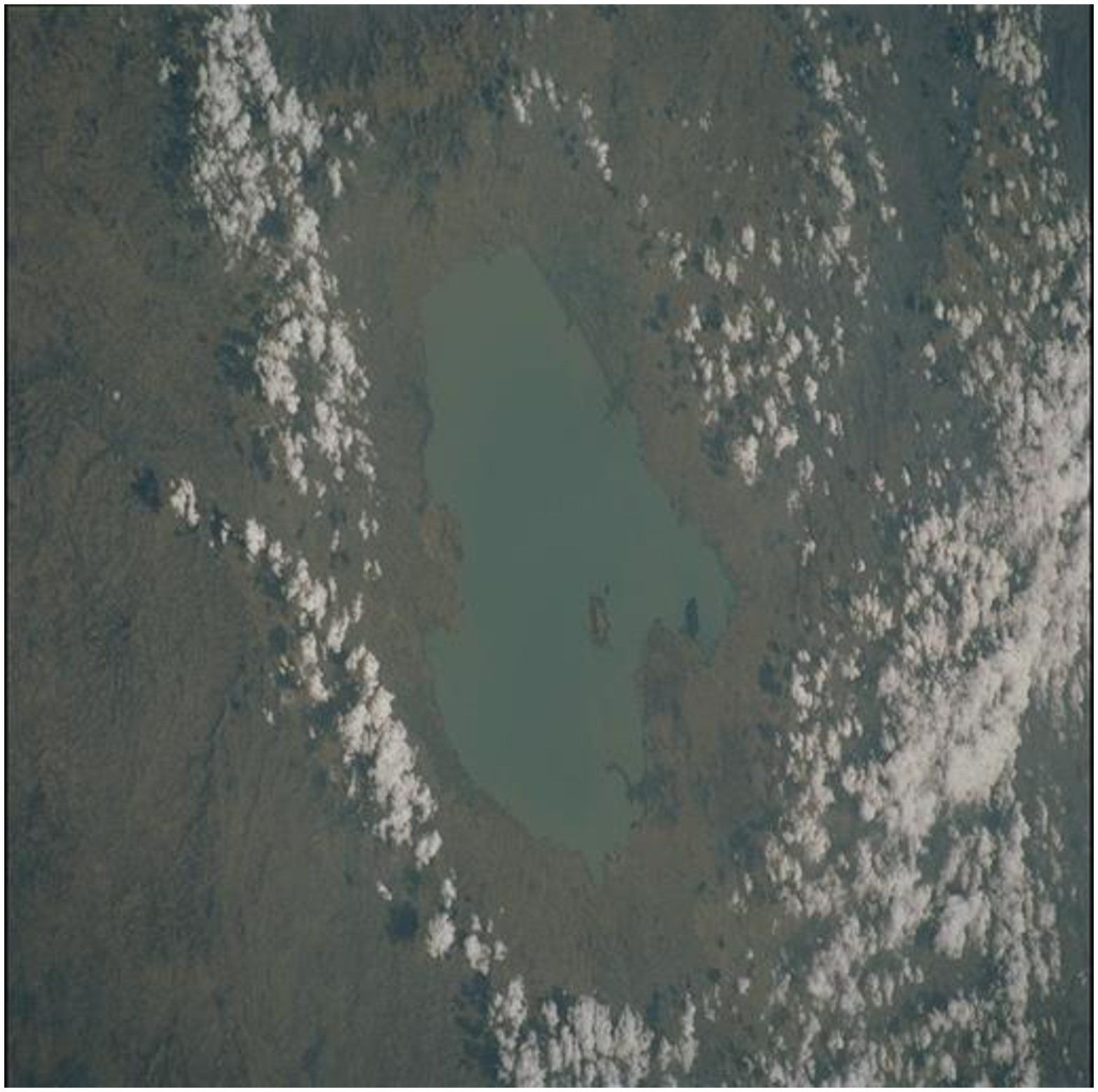} }
\end{center}
\caption{Images courtesy of the Image Science \& Analysis Laboratory, NASA 
Johnson Space Center. For details see text.}
\label{F:Tana}
\end{figure}

Note that the lake's border is rather fuzzy, resulting in a low image gradient. 
The output of edge detection algorithms is degraded even further by the 
substantial cloud cover. Therefore, the border was traced manually by a 
volunteer.
The result is shown in the left-most panel in Figure~\ref{F:data}. 
There are $73$ points along each border curve.


\begin{figure}[hbt]
\begin{center}
\centerline{
\epsfxsize=0.3\hsize
\epsffile{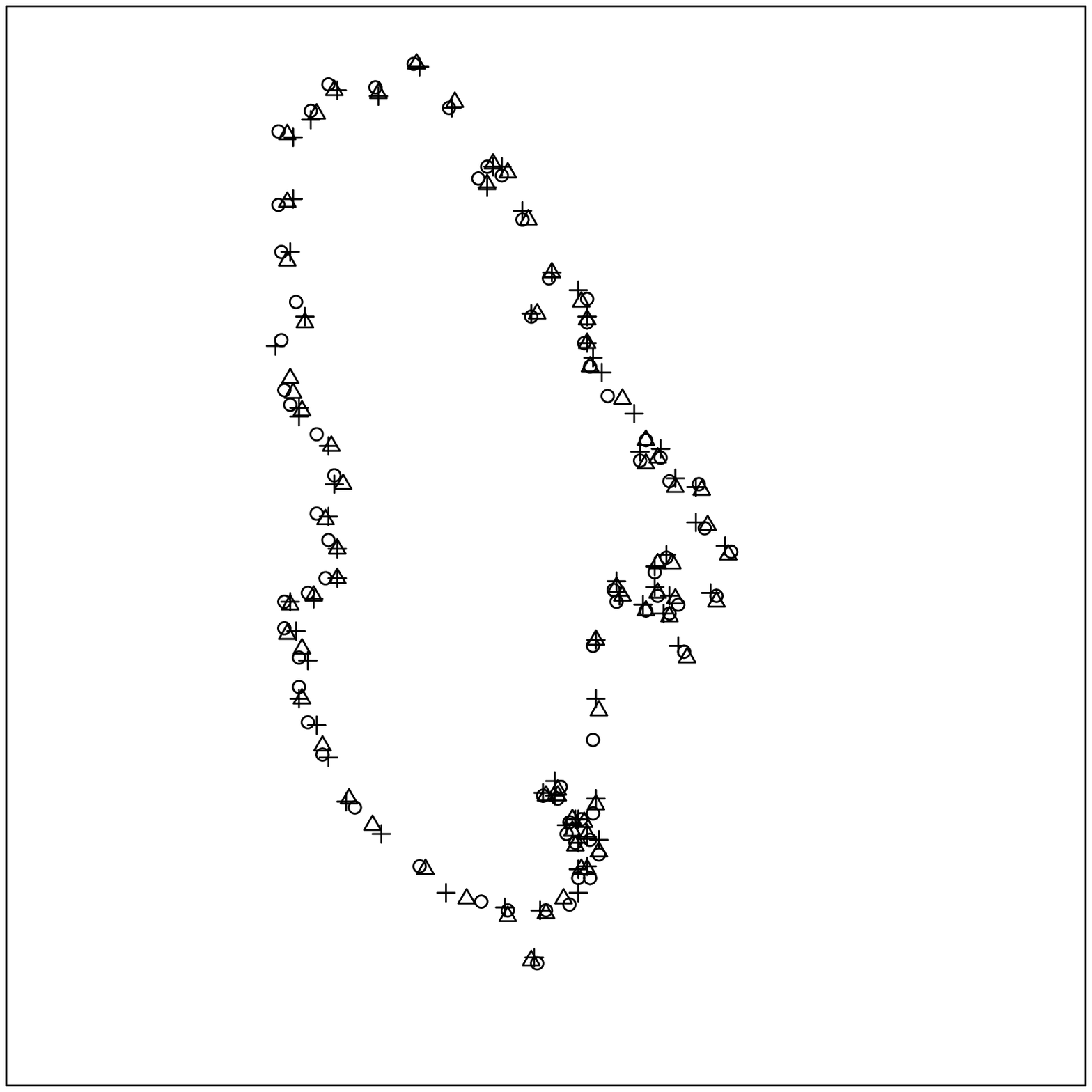} 
\epsfxsize=0.3\hsize
\epsffile{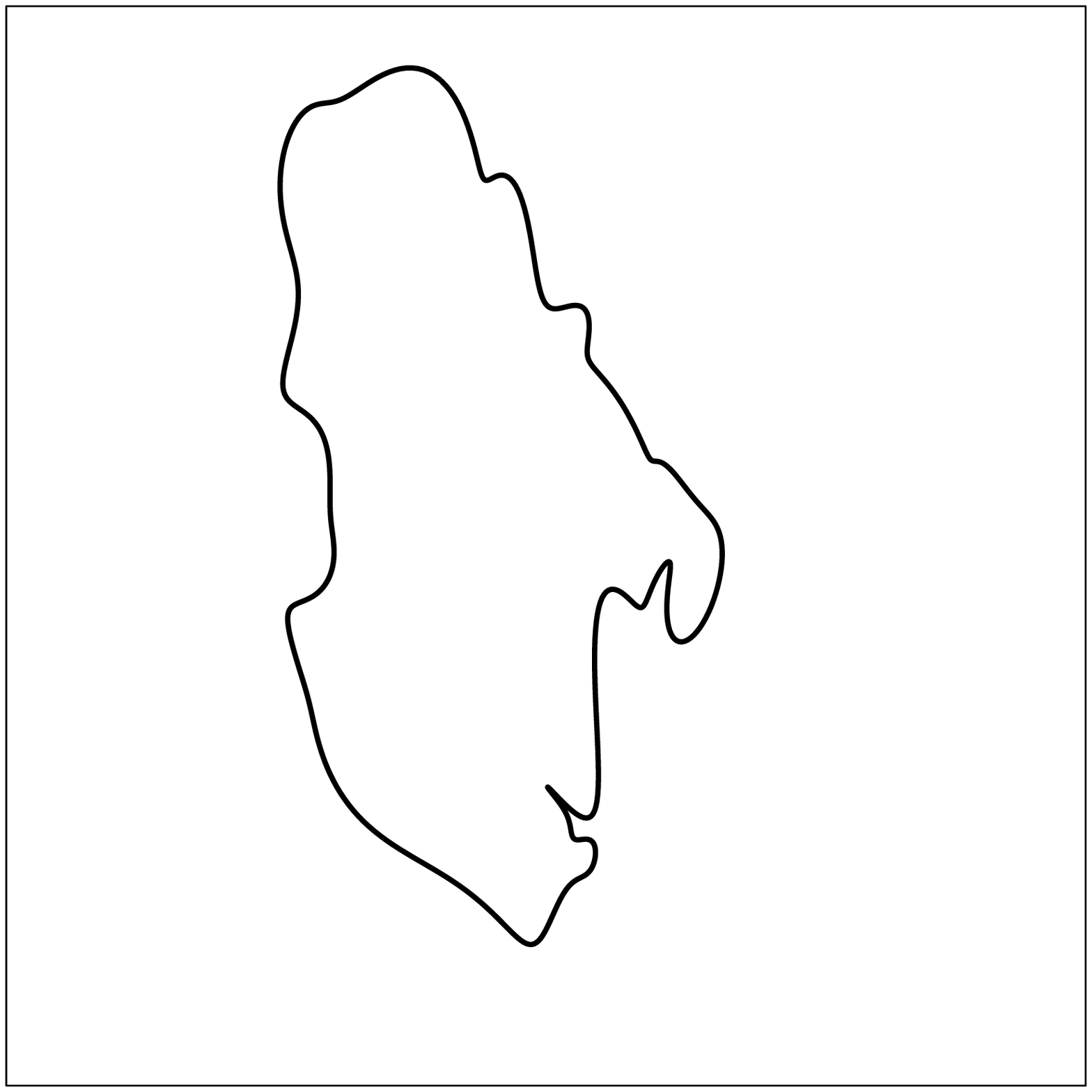} }
\end{center}
\caption{Left panel: Sampled boundary curves corresponding to 
Figure~\ref{F:Tana}. Circles trace the boundary of Lake Tana in the 
left-most panel, triangles correspond to the middle panel, and crosses trace 
the lake boundary in the right-most panel of Figure~\ref{F:Tana}. 
Right panel: Estimated border.}
\label{F:data}
\end{figure}

In contrast to the simulated data considered in the previous subsection,
the curves are not necessarily well aligned. We therefore consider
$M_{73}(\alpha_0, \alpha_1, \alpha_2)$ as in (\ref{e:M}). 
Using $20$ Fourier coefficients and $\alpha_0 = 0$, the optimal 
parameters are $\hat \alpha_1 = - 0.44$ and $\hat \alpha_2 = -2.33$ 
radians. The value of the optimisation function is $1195.048$ 
corresponding to an average error of $2.34$ pixels. The result can be 
improved by including diffeomorphic changes in speed. Optimising
$M_{73}(\alpha_0, \alpha_1, \alpha_2, w_0, w_1, w_t)$ for vectors $w_t$
in $\oR^{2m}$ with $m=5$, cf.~Section~\ref{S:align}, we find an $M$-value
of $568.0997$ corresponding to an average error of $1.61$ pixels. 
The optimal parameters are 
\[
\hat w_1 = 
(0.032, 0.029, 0.037, 0.015, 0.0058, 0.036, 0.016, 0.0096, -0.0080, 0.021)^T
\]
and 
\[
\hat w_2 =
(0.016, 0.037, 0.0081, -0.016, 0.037, 0.031, 0.024, 0.0064, 0.047, 0.029)^T
\]
for the diffeomorphisms and $\hat \alpha_1 = -0.42$ and $\hat \alpha_2 =
-2.32$. 
Finally, the estimated curve is plotted in the right-most panel in
Figure~\ref{F:data}.

\section{Discussion}

In this paper, we formulated a model for objects with uncertain boundaries
using concepts from pattern theory in combination with cyclic Gaussian
processes. The unknown boundary was estimated as a spectral mean by 
carrying out maximum likelihood estimation in the Fourier domain and 
transforming the results back to the spatial domain. We considered the 
integrated squared error and demonstrated how to deal with misalignment 
of the data. Finally, we applied the methods to simulated and real data. 

The approach may be generalised to periodic change models. Indeed,
write $\tau$ for the period. Then we may formulate the model
\begin{equation}
\label{e:curve}
X_{j+ t \tau}^l = X_{j+ t \tau}(\theta_l) =
\Gamma^{(j)}( \varphi_{j+t \tau}(\theta_l-\alpha_{j+t \tau}) ) + 
N_{j+t \tau}^{(j)}(\varphi_{j+t \tau}(\theta_l - \alpha_{j+ t\tau}) )
\end{equation}
for $t = 0, 1, \dots$. Here the $N^{(j)}$ are independent homogeneous 
mean zero cyclic Gaussian noise processes, the $\Gamma^{(j)}$ are unknown 
template curves at $j=0, \dots, \tau - 1$ steps into the period. Since 
the data is periodic, (\ref{e:curve}) splits into $\tau$ submodels of
the form discussed in this paper.

Finally, it is worth noting that, although they are prevalent in shape 
analysis \cite{Youn10}, diffeomorphisms have not been studied much in
stochastic geometry. In this paper, they have been used in different 
roles: for curve modelling and for alignment. It seems to the author
that there is scope for further research concerning the modelling of
random compact sets by means of their boundary curves in light of the
Jordan--Sch\H{o}nflies theorem \cite{Kere04}.

\section*{Acknowledgements}

This research was supported by The Netherlands Organisation for Scientific
Research NWO (613.000.809).


\begin{thebibliography}{99}

\bibitem{AletRuff13}
Aletti, G. and Ruffini, M. (2013).
Is the Brownian bridge a good noise model on the circle?
Technical Report, ArXiv 1210.8245v2, May 2013.
  
\bibitem{Bigo11}
Bigot, J. (2011).
Fr\'echet means of curves for signal averaging and application to
ECG data analysis.
Research report, University of Toulouse.

\bibitem{BurrFran96}
Burrough, P. and Frank, A. (1996).
{\em Geographic objects with indeterminate boundaries\/}.
London: Taylor \& Francis.

\bibitem{ChatColl80}
Chatfield, C. and Collins, A.J. (1980).
{\em Introduction to multivariate analysis\/}.
London: Chapman and Hall.

\bibitem{CoddLevi55}
Coddington, E.A. and Levinson, N. (1955).
{\em Theory of ordinary differential equations\/}.
New~York: McGraw--Hill.

\bibitem{CramLead67}
Cram\`er, H. and Leadbetter, M.R. (1967).
{\em Stationary and related stochastic processes. Sample function properties 
and their applications.}
New~York: Wiley.

\bibitem{Demp67}
Dempster, A.P. (1967).
Upper and lower probabilities induced by a multivalued mapping.
{\em Annals of Mathematical Statistics\/}, 38:325--329.

\bibitem{GoGo81}
Gohberg, I. and Goldberg, S. (1981).
{\em Basic operator theory}.
Boston: Birkh\H{a}user.

\bibitem{GrenMill07}
Grenander, U. and Miller, M.I. (2007).
{\em Pattern theory: from representation to inference.}
Oxford: Oxford University Press.

\bibitem{Hoboetal03}
Hobolth, A., Pedersen, J. and Jensen, E.B.V. (2003).
A continuous parametric shape model.
{\em Annals of the Institute of Statistical Mathematics\/}, 55:227--242.

\bibitem{JonsJens05}
J\'onsd\'ottir, K.Y. and Vedel Jensen, E.B. (2005).
Gaussian radial growth.
{\em Image Analysis and Stereology\/}, 24:117--126.

\bibitem{Kere04}
Keren, D. (2004).
Topologically faithful fitting of simple closed curves.
{\em IEEE Transactions on Pattern Analysis and Machine Intelligence\/},
26:118--123.

\bibitem{Molc05}
Molchanov, I.S. (2005).
{\em Theory of random sets}.
London: Springer.

\bibitem{RogeWill94}
Rogers, L.C.G. and Williams, D. (1994).
{\em Diffusions, Markov processes, and martingales.
Volume One: Foundations}. (Second edition).
Chichester, Wiley.

\bibitem{Shaf76}
Shafer, G. (1976).
{\em Mathematical theory of evidence}. 
Princeton: Princeton University Press.

\bibitem{Soetetal10}
Soetaert, K., Petzoldt, T. and Woodrow Setzer, R. (2010).
Solving differential equations in R: Package deSolve.
{\em Journal of Statistical Software\/}, 33:1--25.

\bibitem{Youn10}
Younes, L. (2010).
{\em Shapes and diffeomorphisms}.
Berlin: Springer.

\bibitem{Zimm01}
Zimmermann, H.-J. (2001).
{\em Fuzzy set theory and its applications}. (Fourth edition).
Dordrecht: Kluwer.
\end{thebibliography}
\end{document}